\documentclass[11pt,draft]{article}
\usepackage{amsmath,amscd}
\usepackage{amssymb,latexsym,amsthm}
\usepackage{color}
\usepackage{amssymb}
\usepackage{color}
\usepackage{amsmath,amsthm,amscd}
\usepackage[latin1]{inputenc}
\usepackage{lscape}
\usepackage{fancyhdr}
\usepackage{amsfonts}
\usepackage{pb-diagram}

\numberwithin{equation}{section}
\newtheorem{theorem}{Theorem}[section]

\newtheorem{proposition}[theorem]{Proposition}

\newtheorem{corollary}[theorem]{Corollary}

\theoremstyle{definition}
\newtheorem{example}[theorem]{Example}
\newtheorem{definition}[theorem]{Definition}
\newtheorem{examples}[theorem]{Examples}

\newtheorem{remark}[theorem]{Remark}

\title{\textbf{Nakayama automorphism of  quasi-commutative skew PBW extensions over AS-regular algebras}}
\author{ H\'ector Su\'arez\footnote{Seminario de \'Algebra Constructiva - SAC2,
Universidad Nacional de Colombia - Sede Bogot\'a.} \footnote{Escuela
de Matem\'aticas y Estad\'{i}stica, Universidad Pedag\'ogica y
Tecnol\'ogica de Colombia - Sede Tunja. } \\Oswaldo Lezama
\footnote{Seminario de \'Algebra Constructiva - SAC2,
Departamento de Matem\'aticas, Universidad Nacional de Colombia - Sede Bogot\'a.}\\Armando Reyes \footnote{Seminario de \'Algebra
Constructiva - SAC2, Departamento de Matem\'aticas, Universidad
Nacional de Colombia - Sede Bogot\'a.}}
\date{}
\begin{document}
\maketitle
\begin{abstract}
\noindent Graded quasi-commutative skew PBW extensions are isomorphic to graded iterated Ore extensions of endomorphism type, whence  graded quasi-commutative skew PBW extensions with coefficients in AS-regular algebras are skew Calabi-Yau and the Nakayama automorphism exists for these extensions. With this in mind, in this paper we give a description of Nakayama automorphism for these non-commutative algebras using the Nakayama automorphism of the ring of the coefficients.

\bigskip


\bigskip

\end{abstract}

\section{Introduction}

Skew PBW extensions and quasi-commutative skew PBW extensions were
defined in \cite{LezamaGallego}. In \cite{Venegas2015} and
\cite{Artamonov1}, respectively, the automorphisms and derivations
of skew PBW extensions were studied. Another properties of  skew PBW
extensions have been studied (see for example \cite{Artamonov1},
\cite{Yesica}, \cite{LezamaAcostaReyes2015}, \cite{LezamaReyes},
\cite{ReyesSuarez2106-2}, \cite{ReyesSuarez2106-3},
\cite{ReyesSuarez2106-4}, \cite{ReyesSuarez2106-5},
\cite{ReyesSuarez2107-4}, \cite{ReyesSuarez2017-3},
\cite{ReyesSuarez2017-2}, \cite{ReyesSuarez2017-1}, \cite{Suarez},
\cite{SuarezLezamaReyes2015}, \cite{SuarezLezamaReyes2107-1},
\cite{SuarezReyes2016} and \cite{SuarezReyes7}). It is known that
quasi-commutative skew PBW extensions are isomorphic to iterated Ore
extensions of endomorphism type (\cite{LezamaReyes}, Theorem 2.3).
For $\mathbb{K}$ a field, in \cite{Suarez}, the first author defined
graded skew PBW extensions and showed that  if $R$ is a finite
presented Koszul $\mathbb{K}$-algebra, then every graded skew PBW
extension of $R$ is Koszul (note that every graded iterated Ore
extension of an Artin-Schelter Regular algebra, AS-regular for
short, is AS-regular, see \cite{Liu1}). Since every graded
quasi-commutative skew PBW extension is isomorphic to a graded
iterated Ore extension of endomorphism type (see Proposition
\ref{prop.baseHilbert quasi}),  we have that if $A$ is a graded
quasi-commutative skew PBW extension of an AS-regular algebra $R$,
then $A$ is AS-regular (see Proposition \ref{prop. AS impl AS}).
Now, for $B$ a connected algebra, $B$ is AS-regular if and only if
$B$ is skew Calabi-Yau (see \cite{Reyes}), and hence the Nakayama
automorphism of AS-regular algebras exists. Therefore, if $R$ is an
AS-regular algebra with Nakayama automorphism $\nu$, then the
Nakayama automorphism $\mu$ of a graded quasi-commutative skew PBW
extension $A$ exists, and we compute it using the Nakayama
automorphism $\nu$ together some especial automorphisms of $R$ and
$A$ (see Theorem \ref{prop.Nakayama}). Note that a skew Calabi-Yau
algebra  is Calabi-Yau if and only if its Nakayama automorphism is
inner.

\section{Graded skew PBW extensions}
\begin{definition}[\cite{LezamaGallego}, Definition 1]\label{def.skewpbwextensions}
Let $R$ and $A$ be rings. We say that $A$ is a \textit{skew PBW
extension over} $R$,  if the following conditions hold:
\begin{enumerate}
\item[\rm (i)]$R\subseteq A$;
\item[\rm (ii)]there exist elements $x_1,\dots ,x_n\in A$ such that $A$ is a left free $R$-module, with basis the basic elements
 ${\rm Mon}(A):= \{x^{\alpha}=x_1^{\alpha_1}\cdots
x_n^{\alpha_n}\mid \alpha=(\alpha_1,\dots ,\alpha_n)\in
\mathbb{N}^n\}$. In this case, it is said also that $A$ is a left polynomial ring over $R$ with respect
to $\{x_1,\cdots, x_n\}$ and Mon(A) is the set of standard monomials of $A$. Moreover,
$x^0_1\cdots x^0_n := 1 \in {\rm Mon}(A)$.

\item[\rm (iii)]For each $1\leq i\leq n$ and any $r\in R\ \backslash\ \{0\}$, there exists an element $c_{i,r}\in R\ \backslash\ \{0\}$ such that $x_ir-c_{i,r}x_i\in R$. 
\item[\rm (iv)]For any natural elements $1\leq i,j\leq n$, there exists $c_{i,j}\in R\ \backslash\ \{0\}$ such that
\begin{equation}\label{sigmadefinicion2}
x_jx_i-c_{i,j}x_ix_j\in R+Rx_1+\cdots +Rx_n.
\end{equation}
Under these conditions, we will write $A:=\sigma(R)\langle
x_1,\dots,x_n\rangle$.
\end{enumerate}
\end{definition}

\begin{remark}\label{rem.grad f y rep} Let $A=\sigma(R)\langle x_1,\dots, x_n\rangle$ be a skew PBW extension with endomorphisms $\sigma_i$, $1\leq i \leq n$,
as in the Proposition \ref{sigmadefinition}. We establish the following notation (see \cite{LezamaGallego}, Definition 6):  $\alpha: = (\alpha_1,\dots, \alpha_n)\in \mathbb{N}^n$; $\sigma^{\alpha}:= (\sigma_1^{\alpha_1}\cdots  \sigma_n^{\alpha_n})$; $|\alpha|:=\alpha_1+\cdots+\alpha_n$; if $\beta: = (\beta_1,\dots, \beta_n)\in \mathbb{N}^n$, then $\alpha +\beta := (\alpha_1+ \beta_1,\dots, \alpha_n +\beta_n)$; for $X = x^{\alpha}=x_1^{\alpha_1}\cdots
x_n^{\alpha_n}$, exp($X$)$:= \alpha$ and deg($X$)$:= |\alpha|$. We have the following properties whose proofs  can be found in \cite{LezamaGallego}, Remark 2 and  Theorem 7.
\begin{enumerate}
\item [\rm (i)] Since ${\rm Mon}(A)$ is a left $R$-basis of $A$, the elements $c_{i,r}$ and $c_{i,j}$ in Definition \ref{def.skewpbwextensions} are unique. In Definition \ref{def.skewpbwextensions} (iv),
$c_{i,i}=1$. This follows from
$x_i^2-c_{i,i}x_i^2=s_0+s_1x_1+\cdots+s_nx_n$, with $s_j\in R$,
which implies $1-c_{i,i}=0=s_j$, for $0\leq j\leq n$.

\item[\rm (ii)]Each element $f\in A\ \backslash\ \{0\}$ has a unique representation as $f=c_1X_1+\cdots+c_tX_t$, with $c_i\in R\ \backslash\ \{0\}$ and $X_i\in {\rm Mon}(A)$,  for $1\leq i\leq t$.
\item[\rm (iii)] For every $x^{\alpha}\in {\rm Mon}(A)$ and every $0\neq r\in  R$, there exist unique elements $r_{\alpha}:=\sigma^{\alpha}(r)\in R \ \backslash\ \{0\}$ and $p_{\alpha, r}\in A$ such that $x^{\alpha}r = r_{\alpha}x^{\alpha} + p_{\alpha, r}$, where $p_{\alpha, r}= 0$ or deg$(p_{\alpha, r}) < |\alpha|$, if $p_{\alpha, r}\neq 0$.
\item[\rm (iv)] For every $x^{\alpha}$, $x^{\beta}\in {\rm Mon}(A)$ there exist unique elements $c_{\alpha,\beta}\in R$ and $p_{\alpha,\beta}\in A$ such
that $x^{\alpha}x^{\beta}= c_{\alpha,\beta}x^{\alpha+\beta} + p_{\alpha,\beta}$ where $c_{\alpha,\beta}$ is left invertible, $p_{\alpha,\beta}= 0$ or deg$(p_{\alpha,\beta}) <|\alpha + \beta|$,  if $p_{\alpha,\beta}\neq 0$.
\end{enumerate}
\end{remark}

\begin{proposition}[\cite{LezamaGallego}, Proposition 3]\label{sigmadefinition}
Let $A$ be a skew PBW extension of $R$. For each $1\leq i\leq
n$, there exist an injective endomorphism $\sigma_i:R\rightarrow
R$ and a $\sigma_i$-derivation $\delta_i:R\rightarrow R$ such that $x_ir=\sigma_i(r)x_i+\delta_i(r),\ r \in R$.
\end{proposition}

The notation $\sigma(R)\langle x_1,\dots,x_n\rangle$ and the name of
the skew PBW extensions are due to the Proposition
\ref{sigmadefinition}. In the following definition we recall  some
sub-classes of skew PBW extensions. Examples of these sub-classes of
algebras can be found in \cite{SuarezReyes7}.
\begin{definition}\label{sigmapbwderivationtype}
Let $A$ be a skew PBW extension of $R$, $\Sigma:=\{\sigma_1,\dotsc, \sigma_n\}$ and $\Delta:=\{\delta_1,\dotsc, \delta_n\}$, where $\sigma_i$ and $\delta_i$ ($1\leq i\le n$) are as in Proposition \ref{sigmadefinition}
\begin{enumerate}
\item[\rm (a)]  $A$ is called {\em pre-commutative}, if the conditions  {\rm(}iv{\rm)} in Definition
\ref{def.skewpbwextensions} are replaced by:\\
 For any $1\leq i,j\leq n$,  there exists $c_{i,j}\in R\ \backslash\ \{0\}$ such that $
x_jx_i-c_{i,j}x_ix_j\in Rx_1+\cdots +Rx_n$.
\item[\rm (b)]\label{def.quasicom} $A$ is called \textit{quasi-commutative}, if the conditions
{\rm(}iii{\rm)} and {\rm(}iv{\rm)} in Definition
\ref{def.skewpbwextensions} are replaced by \begin{enumerate}
\item[\rm (iii')] for each $1\leq i\leq n$ and all $r\in R\ \backslash\ \{0\}$, there exists $c_{i,r}\in R\ \backslash\ \{0\}$ such that
\begin{equation}\label{rel1.quasi}
x_ir=c_{i,r}x_i;
\end{equation}
\item[\rm (iv')]for any $1\leq i,j\leq n$, there exists $c_{i,j}\in R\ \backslash\ \{0\}$ such that
\begin{equation}\label{rel2.quasi}
x_jx_i=c_{i,j}x_ix_j.
\end{equation}
\end{enumerate}
\item[\rm (c)]  $A$ is called \textit{bijective}, if $\sigma_i$ is bijective for each $\sigma_i\in \Sigma$, and $c_{i,j}$ is invertible for any $1\leq
i<j\leq n$.
\item[\rm (d)] If $\sigma_i={\rm id}_R$, for every $\sigma_i\in \Sigma$, we say that $A$ is a skew PBW extension of {\em derivation type}.
\item[\rm (e)]  If $\delta_i = 0$, for every $\delta_i\in \Delta$, we say that $A$ is a skew PBW extension of {\em endomorphism type}.
\item[\rm (f)]   Any element $r$ of $R$ such that $\sigma_i(r)=r$ and $\delta_i(r)=0$ for all $1\leq i\leq n$, will be called a {\em constant}. $A$ is called \emph{constant} if every element of $R$ is constant.
\item[\rm (g)]  $A$ is called {\em semi-commutative} if $A$ is quasi-commutative and constant.
\end{enumerate}
\end{definition}

As we said in the Introduction, the letter $\mathbb{K}$ denotes a field, and  if it is not said otherwise,  every algebra is a $\mathbb{K}$-algebra. The symbol $\mathbb{N}$ will be used to denote the set of natural numbers including zero.

The next proposition was proved by the first author in \cite{Suarez}.

\begin{proposition}\label{prop.grad A}
Let $R=\bigoplus_{m\geq 0}R_m$ be a $\mathbb{N}$-graded algebra and
let $A=\sigma(R)\langle x_1,\dots, x_n\rangle$ be a bijective skew
PBW extension of $R$ satisfying the fo\-llo\-wing two conditions:
\begin{enumerate}
\item[\rm (i)] $\sigma_i$ is a graded ring homomorphism and $\delta_i : R(-1) \to R$ is a graded $\sigma_i$-derivation for all $1\leq i  \leq n$,
where $\sigma_i$ and $\delta_i$ are as in Proposition
\ref{sigmadefinition}.
\item[\rm (ii)]  $x_jx_i-c_{i,j}x_ix_j\in R_2+R_1x_1 +\cdots + R_1x_n$, as in (\ref{sigmadefinicion2}) and $c_{i,j}\in R_0$.
\end{enumerate}
For $p\geq 0$, let $A_p$ the $\mathbb{K}$-space generated by the set
\[\Bigl\{r_tx^{\alpha} \mid t+|\alpha|= p,\  r_t\in R_t \text{  and } x^{\alpha}\in {\rm
Mon}(A)\Bigr\}.
\]
Then $A$ is a $\mathbb{N}$-graded algebra with graduation
 $ A=\bigoplus_{p\geq 0} A_p$. 
\end{proposition}
\begin{proof}
It is clear that $1=x^0_1\cdots x^0_n\in A_0$. Let $f\in A\
\backslash\ \{0\}$, then by Remark \ref{rem.grad f y rep}-(ii), $f$
has a unique representation as $f=r_1X_1+\cdots+r_sX_s$, with
$r_i\in R\ \backslash\ \{0\}$ and $X_i:=x_1^{\alpha_{i_1}}\cdots
x_n^{\alpha_{i_n}}\in {\rm Mon}(A)$ for $1\leq i\leq s$. Let $r_i=
r_{i_{q_1}}+ \cdots + r_{i_{q_m}}$ the unique representation of
$r_i$ in homogeneous elements of $R$. Then $f= (r_{1_{q_1}}+ \cdots
+ r_{1_{q_m}}) x_1^{\alpha_{1_1}}\cdots x_n^{\alpha_{1_n}}+\cdots +
(r_{s_{q_1}}+ \cdots + r_{s_{q_u}})x_1^{\alpha_{s_1}}\cdots
x_n^{\alpha_{s_n}}= r_{1_{q_1}} x_1^{\alpha_{1_1}}\cdots
x_n^{\alpha_{1_n}}+ \cdots + r_{1_{q_m}} x_1^{\alpha_{1_1}}\cdots
x_n^{\alpha_{1_n}}+ \cdots + r_{s_{q_1}}x_1^{\alpha_{s_1}}\cdots
x_n^{\alpha_{s_n}} + \cdots + r_{s_{q_u}}x_1^{\alpha_{s_1}}\cdots
x_n^{\alpha_{s_n}}$ is  the unique representation of $f$ in
homogeneous elements of $A$. Therefore $A$ is a direct sum of the
family $\{A_p\}_{p\geq 0}$ of subspaces of $A$.

Now, let $x\in A_pA_q$. Without loss of generality we can assume
that  $x= (r_tx^{\alpha})(r_sx^{\beta})$ with $r_t\in R_t$,  $r_s\in
R_s$, $ x^{\alpha},  x^{\beta}\in {\rm Mon}(A)$, $t+|\alpha|= p$ and
$s+|\beta|= q$. By Remark \ref{rem.grad f y rep}-(iii), we have that
for $r_s$ and $x^{\alpha}$ there exist unique elements
$r_{s_{\alpha}}:=\sigma^{\alpha}(r_s)\in R \ \backslash\ \{0\}$ and
$p_{\alpha, r_s}\in A$ such that $x= r_t(r_{s_{\alpha}}x^{\alpha} +
p_{\alpha, r_s})x^{\beta}= r_tr_{s_{\alpha}}x^{\alpha}x^{\beta}+
r_tp_{\alpha, r_s}x^{\beta}$, where $p_{\alpha, r_s}= 0$ or
deg$(p_{\alpha, r_s}) < |\alpha|$ if $p_{\alpha, r_s}\neq 0$. Now,
by Remark \ref{rem.grad f y rep}-(iv), we have that for
$x^{\alpha}$, $x^{\beta}$ there exists unique elements
$c_{\alpha,\beta}\in R$ and $p_{\alpha,\beta}\in A$ such that
$x=r_tr_{s_{\alpha}}(c_{\alpha,\beta}x^{\alpha+\beta} +
p_{\alpha,\beta})+ r_tp_{\alpha,
r_s}x^{\beta}=r_tr_{s_{\alpha}}c_{\alpha,\beta}x^{\alpha+\beta} +
r_tr_{s_{\alpha}}p_{\alpha,\beta}+ r_tp_{\alpha, r_s}x^{\beta}$,
where $c_{\alpha,\beta}$ is left invertible, $p_{\alpha,\beta}= 0$
or deg$(p_{\alpha,\beta}) <|\alpha + \beta|$ if
$p_{\alpha,\beta}\neq 0$. We note that:
\begin{enumerate}
\item Since $\sigma_i$ is graded  for $1\leq i\leq n$, then $\sigma_i^{\alpha_i}$ is graded and therefore $\sigma^{\alpha}$ is graded.
Then $r_{s_{\alpha}}:=\sigma^{\alpha}(r_s)\in R_s$ and
$\delta_i^{\alpha_i}(r_s)\in R_{s+\alpha_i}$, for $1\leq i\leq n$
and $\alpha_i\geq 0$.
\item If $W[\delta_i^\nu\sigma_i^{\alpha_{i}-\nu}]$ represents the sum of
the possible words that can be constructed with the alphabet
composed of $\nu$ times the symbol $\delta_i$  and $\alpha_i-\nu$
times the symbol $\sigma_i$,  then $x_i^{\alpha_i}r_s
=\sum_{\nu=0}^{\alpha_i}W[\delta_i^\nu\sigma_i^{\alpha_{i}-\nu}](r_s)x_i^{\alpha_i-\nu}\in
A_{s+\alpha_i}$,  since each summand in the above expression is in
$A_{s+\alpha_i}$.
\item From condition {\rm (ii)}, we have that for $1\leq i< j\leq n$, $x_jx_i=c_{i,j}x_ix_j+ r_{0_{ij}} + r_{1_{ij}}x_1
+ \cdots + r_{n_{ij}}x_n\in A_2$. Then, for  $1\leq i< j< k\leq n$,
we have that

     $ x_k(x_jx_i)= x_k(c_{i,j}x_ix_j+ r_{0_{ij}} + r_{1_{ij}}x_1 + \cdots + r_{n_{ij}}x_n)$\\
     $=  (\sigma_k(c_{i,j})x_kx_ix_j+\delta_k(c_{i,j})x_ix_j) + (\sigma_k(r_{0_{ij}})x_k +\delta_k(r_{0_{ij}}))$\\
     \text{\quad} $+ (\sigma_k(r_{1_{ij}})x_kx_1 +\delta_k(r_{1_{ij}})x_1)+\cdots
    + (\sigma_k(r_{n_{ij}})x_kx_n +\delta_k(r_{n_{ij}})x_n)$\\
     $= \sigma_k(c_{i,j})[c_{i,k}x_ix_k+ r_{0_{ik}} + r_{1_{ik}}x_1 + \cdots + r_{n_{ik}}x_n]x_j +\delta_k(c_{i,j})x_ix_j + \sigma_k(r_{0_{ij}})x_k$\\
      \text{\quad} $+\delta_k(r_{0_{ij}}) + \sigma_k(r_{1_{ij}})[c_{1,k}x_1x_k + r_{0_{1k}} + r_{1_{1k}}x_1 + \cdots + r_{n_{1k}}x_n ]$\\
      \text{\quad} $+\delta_k(r_{1_{ij}})x_1+\cdots + \sigma_k(r_{n_{ij}})x_kx_n+\delta_k(r_{n_{ij}})x_n$\\
        $=  \sigma_k(c_{i,j})c_{i,k}x_i[c_{j,k}x_jx_k + r_{0_{jk}} + r_{1_{jk}}x_1 + \cdots + r_{n_{jk}}x_n] + \sigma_k(c_{i,j}) r_{0_{ik}}x_j$\\
   \text{\quad} $ + \sigma_k(c_{i,j}) r_{1_{ik}}x_1x_j + \cdots  + \sigma_k(c_{i,j}) r_{n_{ik}}x_nx_j  +\delta_k(c_{i,j})x_ix_j + \sigma_k(r_{0_{ij}})x_k$\\
    \text{\quad} $+\delta_k(r_{0_{ij}}) + \sigma_k(r_{1_{ij}})[c_{1,k}x_1x_k  + r_{0_{1k}} + r_{1_{1k}}x_1 + \cdots   + r_{n_{1k}}x_n ]$\\
     \text{\quad} $ +\delta_k(r_{1_{ij}})x_1+\cdots      + \sigma_k(r_{n_{ij}})x_kx_n+\delta_k(r_{n_{ij}})x_n$\\
     $= \sigma_k(c_{i,j})c_{i,k}\sigma_i(c_{j,k})x_ix_jx_k+ \sigma_k(c_{i,j})c_{i,k}\delta_i(c_{j,k})x_jx_k + \sigma_k(c_{i,j})c_{i,k}\sigma(c_{j,k})\sigma_i(r_{0_{ij}})x_i$ \\
   \text{\quad} $  + \sigma_k(c_{i,j})c_{i,k}\delta_i(r_{0_{ij}}) + \sigma_k(c_{i,j})c_{i,k}\sigma_i(r_{1_{jk}})x_ix_1 +   \sigma_k(c_{i,j})c_{i,k}\delta_i(r_{1_{jk}})x_1+ \cdots $\\
     \text{\quad} $  + \sigma_k(c_{i,j})c_{i,k}\sigma_i(r_{n_{jk}})x_ix_n + \sigma_k(c_{i,j})c_{i,k}\delta_i(r_{n_{jk}})x_n + \sigma_k(c_{i,j}) r_{0_{ik}}x_j$\\
    \text{\quad} $ +  \sigma_k(c_{i,j}) r_{1_{ik}}x_1x_j + \cdots + \sigma_k(c_{i,j}) r_{n_{ik}}x_nx_j +\delta_k(c_{i,j})x_ix_j + \sigma_k(r_{0_{ij}})x_k  +\delta_k(r_{0_{ij}}) $\\
    \text{\quad} $ + \sigma_k(r_{1_{ij}})c_{1,k}x_1x_k + \sigma_k(r_{1_{ij}})r_{0_{1k}} +\sigma_k(r_{1_{ij}}) r_{1_{1k}}x_1 + \cdots + \sigma_k(r_{1_{ij}})r_{n_{1k}}x_n $ \\
    \text{\quad} $ +\delta_k(r_{1_{ij}})x_1+\cdots + \sigma_k(r_{n_{ij}})x_kx_n+\delta_k(r_{n_{ij}})x_n$\\
     $=  \sigma_k(c_{i,j})c_{i,k}\sigma_i(c_{j,k})x_ix_jx_k+ \sigma_k(c_{i,j})c_{i,k}\delta_i(c_{j,k})x_jx_k + \sigma_k(c_{i,j})c_{i,k}\sigma(c_{j,k})\sigma_i(r_{0_{ij}})x_i $\\
   \text{\quad} $   + \sigma_k(c_{i,j})c_{i,k}\delta_i(r_{0_{ij}}) + \sigma_k(c_{i,j})c_{i,k}\sigma_i(r_{1_{jk}})c_{1,i}x_1x_i + \sigma_k(c_{i,j})c_{i,k}\sigma_i(r_{1_{jk}})r_{0_{1i}}$\\
     \text{\quad} $ + \sigma_k(c_{i,j})c_{i,k}\sigma_i(r_{1_{jk}})r_{1_{1i}}x_1 + \cdots   + \sigma_k(c_{i,j})c_{i,k}\sigma_i(r_{1_{jk}}) r_{n_{1i}}x_n$ \\
    \text{\quad} $   +    \sigma_k(c_{i,j})c_{i,k}\delta_i(r_{1_{jk}})x_1+ \cdots   + \sigma_k(c_{i,j})c_{i,k}\sigma_i(r_{n_{jk}})x_ix_n +
    \sigma_k(c_{i,j})c_{i,k}\delta_i(r_{n_{jk}})x_n$\\
     \text{\quad} $ + \sigma_k(c_{i,j}) r_{0_{ik}}x_j  +  \sigma_k(c_{i,j}) r_{1_{ik}}x_1x_j + \cdots +  \sigma_k(c_{i,j}) r_{n_{ik}}c_{j,n}x_jx_n +
      \sigma_k(c_{i,j}) r_{n_{ik}}\delta_n(x_j)$\\
     \text{\quad} $ + \sigma_k(c_{i,j}) r_{n_{ik}}r_{0_{jn}}+\sigma_k(c_{i,j}) r_{n_{ik}}r_{1_{jn}}x_1+\cdots + \sigma_k(c_{i,j}) r_{n_{ik}}r_{n_{jn}}x_n + \delta_k(c_{i,j})x_ix_j$\\
      \text{\quad} $ + \sigma_k(r_{0_{ij}})x_k  +\delta_k(r_{0_{ij}})
     + \sigma_k(r_{1_{ij}})c_{1,k}x_1x_k + \sigma_k(r_{1_{ij}})r_{0_{1k}} +\sigma_k(r_{1_{ij}}) r_{1_{1k}}x_1 + \cdots$\\
      \text{\quad} $ + \sigma_k(r_{1_{ij}})r_{n_{1k}}x_n
      + \delta_k(r_{1_{ij}})x_1+ \cdots + \sigma_k(r_{n_{ij}})x_kx_n+\delta_k(r_{n_{ij}})x_n$.

Since all summands in the above sum have the form $rx$, where $r$ is
an homogeneous element of $R$, $x\in {\rm Mon}(A)$ and $rx\in A_3$,
we have that  $x_kx_jx_i\in A_3$. Following this procedure we get in
general that $x_{i_1}x_{i_2}\cdots x_{i_m}\in A_m$ for $1\leq
i_k\leq n$, $1\leq k \leq m$, $m\geq 1$.
\item In a similar way and following the proof of Theorem 7, in \cite{LezamaGallego}, we obtain that
$p_{\alpha, r_s}\in  A_{|\alpha| +s}$ and $p_{\alpha,\beta}\in
A_{|\alpha|+|\beta|}$. Then
$r_tr_{s_{\alpha}}c_{\alpha,\beta}x^{\alpha+\beta}\in
A_{t+s+|\alpha|+|\beta|}$, $r_tr_{s_{\alpha}}p_{\alpha,\beta}\in
A_{t+s+|\alpha|+|\beta|}$ and $r_tp_{\alpha, r_s}x^{\beta}\in
A_{t+|\alpha|+ s+|\beta|}$, i.e., $x\in A_{p+q}$.
\end{enumerate}
\end{proof}

\begin{definition}[\cite{Suarez}, Definition 2.6]\label{def. graded skew PBW ext} Let  $A=\sigma(R)\langle x_1,\dots, x_n\rangle$ be a bijective
skew PBW extension of a  $\mathbb{N}$-graded algebra
$R=\bigoplus_{m\geq 0}R_m$. We say that $A$ is a \emph{graded  skew
PBW extension} if $A$ satisfies the conditions (i) and (ii) in
Proposition \ref{prop.grad A}.
\end{definition}
Note that the family of  graded iterated Ore extensions is strictly contained in the family  of graded skew PBW extensions (see \cite{Suarez}, Remark 2.11).

\begin{proposition}\label{prop.quasi is grad}
Quasi-commutative skew PBW extensions with  the trivial graduation of  $R$ are graded  skew PBW extensions.  If we assume that  $R$  has a different graduation to the trivial graduation, then $A$ is graded skew PBW extension  if and only if $\sigma_i$ is graded and $c_{i,j}\in R_0$,  for $1\leq i,j \leq n$.
\end{proposition}

\begin{proof}
Let $R=R_0$ and $r\in R=R_0$. From (\ref{rel1.quasi}) we have that $x_ir=c_{i,r}x_i= \sigma_i(r)x_i$. So, $\sigma_i(r)=c_{i,r}\in R= R_0$ and $\delta_i=0$, for $1\leq i\leq n$. Therefore $\sigma_i$ is a graded ring homomorphism and $\delta_i : R(-1) \to R$ is a graded $\sigma_i$-derivation for all $1\leq i  \leq n$. From (\ref{rel2.quasi}) we have that $x_jx_i-c_{i,j}x_ix_j=0 \in R_2+R_1x_1 +\cdots + R_1x_n$ and $c_{i,j}\in R= R_0$. If $R$ has a nontrivial graduation, then the result  is obtained from de relations (\ref{rel1.quasi}), (\ref{rel2.quasi}) and Definition \ref{def. graded skew PBW ext}.
\end{proof}

\begin{examples}\label{ex.quasi} We present some  examples of graded quasi-commutative skew PBW extensions.
\begin{enumerate}
\item\label{Sklyanin} The  \emph{Sklyanin algebra} is the algebra
$S = \mathbb{K}\langle x, y, z\rangle/\langle ayx + bxy + cz^2, axz + bzx + cy^2, azy + byz + cx^2\rangle$, where $a, b, c\in \mathbb{K}$. If $c\neq 0$ then $S$ is not a skew PBW extension. If $c=0$ and $a,b\neq 0$ then in $S$: $yx = -\frac{b}{a}xy$; $zx =  -\frac{a}{b}xz$ and $zy =  -\frac{b}{a}yz$, therefore $S\cong\sigma(\mathbb{K})\langle x,y,z\rangle$ is a skew PBW extension of $\mathbb{K}$, and we call this algebra a \emph{particular Sklyanin algebra}. The particular Sklyanin algebra is graded quasi-commutative skew PBW extension of $\mathbb{K}$.
\item \label{ex.shift oper} Let $h\in \mathbb{K}$. The \emph{algebra of shift operators} is defined by $S_h:= \mathbb{K}[t][x_h; \sigma_h]$, where $\sigma_h(p(t)):=p(t-h)$. Notice that $x_ht = (t - h)x_h$ and for  $p(t) \in \mathbb{K}[t]$ we have $x_hp(t)=p(t-ih)x_h^i$. Thus, $S_h\cong \sigma(\mathbb{K}[t])\langle x_h\rangle$ is a graded quasi-commutative skew PBW extension of $\mathbb{K}[t]$, where $\mathbb{K}[t]$  is endowed with trivial graduation.
\item \label{ex.q-dilat oper} For a fixed $q \in \mathbb{K} - \{0\}$, the \emph{algebra of
linear partial $q$-dilation operators $H$}, with polynomial
coefficients,  is the free algebra $\mathbb{K}\langle t_1,\dots
,t_n, H^{(q)}_1, \dots, H^{(q)}_m\rangle$, $n \geq m$, subject to
the relations: $$t_jt_i = t_it_j,\quad 1 \leq i < j \leq n;$$
$$H^{(q)}_i t_i = qt_iH^{(q)}_i,\quad 1 \leq i \leq m;$$ $$H^{(q)}_j
t_i = t_iH^{(q)}_j,\quad i \neq j;$$  $$H^{(q)}_j H^{(q)}_i =
H^{(q)}_iH^{(q)}_j, \quad 1 \leq i < j\leq m.$$ The algebra $H$ is a
graded quasi-commutative skew PBW extension of
$\mathbb{K}[t_1,\dots, t_n]$, where $\mathbb{K}[t_1,\dots, t_n]$ is
endowed with usual graduation.
\item \label{ex.multipl Weyl} The
\emph{quantum polynomial ring} $\mathcal{O}_n(\lambda_{ji})$ is the
algebra  generated by $x_1,\dots,x_n$ subject to the relations:
$x_jx_i =\lambda_{ji}x_ix_j ,\ 1\leq i<j\leq n$, $\lambda_{ji}\in
\mathbb{K}\setminus \{0\}$. Thus
$\mathcal{O}_n(\lambda_{ji})\cong\sigma(\mathbb{K})\langle
x_1,\dotsc,x_n\rangle\cong \sigma(\mathbb{K}[x_1])\langle
x_2,\dotsc,x_n\rangle$.
\item \label{mult param quant aff}  Let $n\geq 1$ and $\textbf{q}$ be a  matrix $(q_{ij})_{n\times n}$ whit entries in a field $\mathbb{K}$  where
 $q_{ii} = 1$ y $q_{ij}q_{ji} = 1$ for all $1\leq i, j \leq n$. Then \emph{multi-parameter quantum affine $n$-space} $\mathcal{O}_{\textbf{q}}(\mathbb{K}^n)$ is defined  to be $\mathbb{K}-$algebra generated by
 $x_1,\cdots, x_n$ with the relations $x_jx_i = q_{ij}x_ix_j$, for all $1\leq i, j \leq n$.
\end{enumerate}
\end{examples}

Examples of graded skew PBW extensions  over commutative polynomial rings $R$ which are not quasi-commutative, and where $R$ has the usual graduation, can be found in \cite{Suarez}.

\begin{remark}[\cite{Suarez}, Remark 2.10]\label{rem.prop of graded skew} Let $A=\sigma(R)\langle x_1,\dots, x_n\rangle$ be a graded skew PBW extension. Then  we  have the following properties:
\begin{enumerate}
\item[\rm (i)] $A$ is a $\mathbb{N}$-graded algebra and  $A_0=R_0$.
\item[\rm (ii)] $R$ is connected if and only if $A$ is connected.
\item[\rm (iii)] If $R$ is finitely generated then $A$ is finitely generated.
\item[\rm(iv)] For {\rm (i)}, {\rm (ii)} and {\rm (iii)} above,  we have that if $R$ is a finitely graded algebra then $A$ is a finitely graded algebra.
\item[\rm (v)] If $R$ is locally finite, then $A$ as $\mathbb{K}$-algebra is a locally finite.
\item[\rm (vi)] $A$ as $R$-module  is locally finite.
\item[\rm(vii)] If $A$ is quasi-commutative and $R$ is concentrate in degree 0, then $A_0=R$.
\item[\rm(viii)] If $R$ is a homogeneous quadratic algebra then $A$ is a homogeneous quadratic algebra.
\item[\rm(ix)] If $R$ is finitely presented then $A$ is finitely presented.
\end{enumerate}
\end{remark}

\begin{proposition}[\cite{LezamaReyes}, Theorem 2.3]\label{1.3.3}
Let $A$ be a quasi-commutative skew PBW extension of a ring $R$.
\begin{enumerate}
\item[\rm (i)] $A$ is isomorphic to an
iterated Ore extension of endomorphism type $R[z_1;
\theta_1]\cdots[z_n; \theta_n]$, where $\theta_1=\sigma_1$; $\theta_j: R[z_1;\theta_1]\cdots[z_{j-1};\theta_{j-1}]\to
R[z_1;\theta_1]\cdots[z_{j-1};\theta_{j-1}]$ is such that
$\theta_j(z_i)=c_{i,j}z_i$ ($c_{i,j}\in R$ as in
(\ref{sigmadefinicion2})), $1 \leq i < j \leq n$ and
$\theta_i(r)=\sigma_i(r)$, for $r\in R$.
\item[\rm (ii)] If $A$ is bijective, then each $\theta_i$ in {\rm (i)} is  bijective.
\end{enumerate}
\end{proposition}
The following proposition  establishes the relation between graded
skew PBW extensions and graded iterated Ore extensions.

\begin{proposition}\label{prop.baseHilbert quasi} Let $A=\sigma(R)\langle x_1,\dots, x_n\rangle$ be a graded quasi-commutative skew PBW extension of $R$.
Then $A$  is isomorphic to a graded iterated Ore extension of
endomorphism type.
\end{proposition}
\begin{proof}
By Proposition \ref{1.3.3}-(i) we have that $A$ is isomorphic to an
iterated Ore extension of endomorphism type $R[z_1;
\theta_1]\cdots[z_n; \theta_n]$, where $\theta_1=\sigma_1$,
$$\theta_j: R[z_1;\theta_1]\cdots[z_{j-1};\theta_{j-1}]\to
R[z_1;\theta_1]\cdots[z_{j-1};\theta_{j-1}]$$ is such that
$\theta_j(z_i)=c_{i,j}z_i$ ($c_{i,j}\in R$ as in
(\ref{sigmadefinicion2})), $1 \leq i < j \leq n$ and
$\theta_i(r)=\sigma_i(r)$, for $r\in R$.  Since $A$ is bijective
then by Proposition \ref{1.3.3}-(ii) $\theta_i$ is bijective. Since
$A$ is graded then $\sigma_i$ is graded and $c_{i,j}\in R_0$.
Moreover, since $\theta_i(r)=\sigma_i(r)$, then $\theta_i$ is a
graded automorphism for each $i$. Note that $z_i$ has graded 1 in
$A$, for all $i$. Thus, $A\cong R[z_1; \theta_1]\cdots[z_n;
\theta_n]$ is  a graded iterated Ore extension.
\end{proof}
\section{AS-Regular and Koszul algebras}

Let $B = \mathbb{K}\oplus B_1\oplus B_2\oplus \cdots$ be a finitely presented graded algebra over  $\mathbb{K}$. The algebra $B$ will be called \emph{AS-regular},  if $B$ has the
following properties:
\begin{enumerate}
\item [(i)] $B$ has finite global dimension $d$: every graded $B-$module has projective
dimension $\leq d$.
\item [(ii)] $B$ has finite $GK$-dimension.
\item  [(iii)] $B$ is \emph{Gorenstein}, meaning that ${\rm Ext}_B^i(\mathbb{K},
B) =0$ if $i \neq d$, and ${\rm Ext}^d_B(\mathbb{K}, B)\cong \mathbb{K}$
\end{enumerate}

\begin{proposition}[\cite{SuarezLezamaReyes2107-1}, Theorem 17]\label{prop. AS impl AS} Let $A=\sigma(R)\langle x_1, \dots, x_n\rangle$ be a graded quasi-commutative skew PBW extension.  If $R$ is AS-regular, then $A$ is AS-regular.
\end{proposition}
Let $B$ be a finitely graded algebra and let $M$, $N$ be $\mathbb{Z}$-graded
$B$-modules. Then there is a natural inclusion $\underline{{\rm Hom}}_B(M,N)\hookrightarrow {\rm Hom}_B(M,N)$. If $M$ is an $B$-module  finitely generated, then
$\underline{{\rm Hom}}_B(M,N)\cong {\rm Hom}_B(M,N)$ and $\underline{{\rm Ext}}^i_B(M,N)\cong {\rm Ext}^i_B(M,N)$. A graded algebra $B$ is \emph{quadratic} if $B = T(V)/\langle R\rangle$ where $V$ is a finite dimensional $\mathbb{K}$-vector space concentrated in degree l, $T(V)$ is the tensor algebra on $V$, with the induced grading and $\langle R\rangle$ is the ideal generated by a subspace $R\subseteq V\otimes V$. The \emph{dual} of such a quadratic algebra is $B^{!}:= T(V^{*})/\langle R^{\bot}\rangle$, where $R^{\bot} = \{\lambda\in V^{*}\otimes V^{*}\mid \lambda(r) = 0 \text{ for all } r\in R\}$. We identify $(V\otimes V)^{*}$ with $V^{*}\otimes V^{*}$ by defining $(\alpha\otimes\beta)(u\otimes v):=\alpha(u)\beta(v)$ for $\alpha,\beta\in V^{*}$ and $u, v\in V$.\\

Let $B=\mathbb{K}\bigoplus B_1\bigoplus B_2\bigoplus\cdots$ be a locally finite graded algebra  and
$E(B) =\bigoplus_{s,p}E^{s,p}(B) = \bigoplus_{s,p}{\rm Ext}^{s,p}_B(\mathbb{K}, \mathbb{K})$ the
associated bigraded Yoneda algebra, where $s$ is the cohomology degree  and $-p $ is the internal
degree inherited from the grading on $B$. Let $E^{s}(B) = \bigoplus_{p}E^{s,p}(B)$. $B$  is said to be
\emph{Koszul} if the following equivalent conditions hold:
\begin{enumerate}
\item[\rm (i)]  ${\rm Ext}^{s,p}_B(\mathbb{K}, \mathbb{K})=0$ for $s\neq p$;
\item[\rm (ii)] $B$ is one-generated and the algebra  ${\rm Ext}^*_B (\mathbb{K}, \mathbb{K})$ is generated by  ${\rm Ext}_B^{1}(\mathbb{K}, \mathbb{K})$, i.e., $E(B)$ is generated in the first cohomological degree;
\item[\rm (iii)] The module $\mathbb{K}$ admits a \emph{linear free resolution}, i.e., a resolution by free $B$-modules
\[
\cdots \to P_2\to P_1\to P_0\to \mathbb{K}\to 0,
\]
such that $P_i$ is generated in degree $i$.
\item[\rm(iv)] $\underline{{\rm Ext}}^{*}(\mathbb{K}, \mathbb{K}) \cong B^{!}$ as graded $\mathbb{K}$-algebras.
\end{enumerate}

The following theorem was proved by the first author in \cite{Suarez}.
\begin{theorem}\label{teo.Kosz impl Kosz} Let $A=\sigma(R)\langle x_1, \dots, x_n\rangle$ be a  graded skew PBW extension.
\begin{enumerate}
\item[\rm (i)] The graded iterated Ore extension $A := R[x_1; \sigma_1,\delta_1]\cdots [x_{n}; \sigma_{n},\delta_{n}]$ is Koszul if and only if $R$ is Koszul {\rm (\cite{Suarez}}, Proposition 3.1{\rm )}.
\item[\rm (ii)]  If $A$ is quasi-commutative, then $R$ is  Koszul  if and only if $A$  is Koszul {\rm (\cite{Suarez}}, Proposition 3.3{\rm )}.
\item[\rm (iii)] Let $R$ be  a finitely presented algebra. If  $R$ is a PBW
algebra then $A$ is Koszul algebra {\rm (\cite{Suarez}}, Corollary 4.4{\rm )}.
\item[\rm (iv)] If $R$ is a finitely presented Koszul algebra, then  $A$ is Koszul {\rm (\cite{Suarez},}  Theorem 5.5{\rm )}.
\end{enumerate}
\end{theorem}

\section{Skew Calabi-Yau algebras and Nakayama automorphism}

The enveloping algebra of a ring $B$ is defined as $B^e := B\otimes
B^{op}$. We characterize the enveloping algebra of a skew PBW
extension in \cite{ReyesSuarez2107-4}. If $M$ is an $B$-bimodule,
then $M$ is an $B^e$ module with the action given by $(a\otimes
b)\cdot m = amb$, for all $m\in M$, $a, b \in B$. Given
automorphisms $\nu, \tau \in {\rm Aut}(B)$, we can define the
twisted $B^e$-module ${^{\nu}M^{\tau}}$ with the rule $(a\otimes b)
\cdot m = \nu(a)m\tau(b)$, for all $m\in M$, $a, b \in B$. When one
or the other of $\nu$, $\tau$ is the identity map, we shall simply
omit it, writing for example $M^{\nu}$ for ${^{1} M^{\nu}}$.

\begin{proposition}[\cite{Goodman}, Lemma 2.1]\label{lemma2.1 Untwisting}
Let $\nu$, $\sigma$ and  $\phi$ be automorphisms of $B$. Then
\begin{enumerate}
\item [\rm (i)] The map ${^{\nu}B^{\sigma}}\to {^{\phi\nu}B^{\phi\sigma}},\ a\mapsto \phi(a)$ is an isomorphism of $B^e$-modules. In particular,
$${^{\nu} B^{\sigma}}\cong B^{{\nu}^{-1}\sigma} \cong {^{\sigma^{-1}\nu} B} \text{ and } B^{\sigma}\cong {^{\sigma^{-1}}B}.$$
\item [\rm (ii)]  $B\cong B^{\sigma}$ as  $B^e$-modules if and only if $\sigma$ is an inner automorphism.
\end{enumerate}
\end{proposition}

An algebra $B$ is said to be \emph{homologically smooth}, if  as an $B^e$-module, $B$ has a finitely generated projective resolution of finite length. The length of this resolution is known as the
\emph{Hochschild dimension} of $B$.

\begin{definition}\label{def1.1} A graded algebra $B$ is said to be \emph{skew Calabi-Yau} of dimension $d$ if
\begin{enumerate}

\item[\rm (i)] $B$ is homologically smooth.
\item[\rm (ii)] There exists an  algebra automorphism  $\nu$ of $B$ such that \[Ext^i_{B^e} (B,B^e) \cong \left\{
                                         \begin{array}{ll}
                                           0, & i\neq d; \\
                                           B^{\nu}(l), & i= d.
                                         \end{array}
                                       \right.\]
 as $B^e$ -modules, for some integer $l$.  If $\nu$ is the identity, then $B$ is said to be \emph{Calabi-Yau}.
\end{enumerate}
\end{definition}

Sometimes condition (ii) is called the skew Calabi-Yau condition.\\

The skew Calabi-Yau condition appears to have first been explicitly defined in \cite{Brown} under the term {\em rigid Gorenstein}. The automorphism $\nu$ is called the \emph{Nakayama} automorphism of $B$, and is unique up to inner automorphisms of $B$. As a consequence of Proposition \ref{lemma2.1 Untwisting}, we have that a skew Calabi-Yau algebra is Calabi-Yau  if and only if its Nakayama automorphism is inner. If $B$ is a Calabi-Yau algebra of dimension $d$, then the Hochschild dimension of $B$ (that is, the projective dimension of $A$ as an $A$-bimodule) is $d$ (see \cite{Berger3}, Proposition 2.2). Moreover,  the Hochschild dimension of $B$  coincides with the global dimension of $B$ (see \cite{Berger3}, Remark 2.8).

\begin{proposition}\label{Nak.uniq}
Let $B$ be a skew Calabi-Yau algebra with Nakayama automorphism
$\nu$. Then $\nu$ is unique up to an inner automorphism, i.e, the
Nakayama automorphism is determined up to multiplication by an inner
automorphism of $B$.
\end{proposition}
\begin{proof}
Let $B$ be a skew Calabi-Yau algebra with Nakayama automorphism
$\nu$ and let $\mu$ another Nakayama automorphism, i.e.,
$Ext^d_{B^e} (B,B^e) \cong B^{\mu}$, then $Ext^d_{B^e} (B,B^e) \cong
B^{\nu} \cong B^{\mu}$ as $B^e$ -modules. By Proposition
\ref{lemma2.1 Untwisting}-(i), $ B\cong B^{\nu^{-1}\mu}$; by
Proposition  \ref{lemma2.1 Untwisting}-(ii), $\nu^{-1}\mu$ is an
inner automorphism of $B$. Let $\nu^{-1}\mu=\sigma$ where $\sigma$
is an inner automorphism of $B$, so $\mu = \nu \sigma$  for some
inner automorphism $\sigma$ of $B$.
\end{proof}

\begin{proposition}[\cite{Reyes}, Lemma 1.2]\label{prop.lema1.2}
Let $B$ be a connected graded algebra. Then $B$ is skew Calabi-Yau  if and only if it is AS-regular.
\end{proposition}
\begin{proposition}\label{prop. ASKoszulCY} Let $R$ be a Koszul Artin-Schelter regular algebra of global dimension $d$ with  Nakayama automorphism $\nu$.
\begin{enumerate}
\item[\rm (i)]{\rm (\cite{He5}, Theorem 3.3)}  The skew polynomial algebra $B=R[x;\nu]$ is a Calabi-Yau algebra of dimension $d+1$.
\item[\rm (ii)]{\rm (\cite{Zhu}, Remark 3.13)}  There exists a unique skew polynomial extension $B$ such that $B$ is Calabi-Yau.
\item[\rm (iii)]{\rm (\cite{Zhu}, Theorem 3.16)}  If $\sigma$ is a graded algebra automorphism of $R$, then $B = R[x; \sigma]$ is Calabi-Yau if and only if $\sigma = \nu$.
\end{enumerate}
\end{proposition}

The following theorem can also be found in
\cite{SuarezLezamaReyes2107-1}.

\begin{theorem}\label{teo.skew imp skew} Every graded quasi-commutative skew PBW extension $A=\sigma(R)\langle x_1,\dots, x_n\rangle$ of a finitely presented  skew  Calabi-Yau  algebra $R$ of global dimension
$d$, is skew Calabi-Yau of global dimension $d+n$. Moreover, if $R$
is Koszul and
 $\theta_i$ is the Nakayama automorphism of $R[x_1;\theta_1]\cdots [x_{i-1};\theta_{i-1}]$ for $1\leq i\leq n$, then $A$ is Calabi-Yau of dimension $d+n$ ($\theta_i$ as in Proposition \ref{prop.baseHilbert quasi}-(ii), $x_0=1$).
\end{theorem}

\begin{proof}
Since $R$ is connected and skew Calabi-Yau, then by Proposition
\ref{prop.lema1.2} we know that $R$ is Artin-Schelter regular. From
Proposition \ref{prop. AS impl AS} we have that $A$ is
Artin-Schelter regular and, in particular, connected. Thus, using
again Proposition \ref{prop.lema1.2}, we have that $A$ is a skew
Calabi-Yau algebra. By the proof of Proposition \ref{prop. AS impl
AS} we have that the global dimension of $A$ is $d+n$.

For the second part, we know that graded Ore extensions of Koszul
algebras are Koszul algebras and, as a particular case of
Proposition \ref{prop. AS impl AS}, we have that a graded Ore
extension of an Artin-Schelter regular algebra is an Artin-Schelter
regular algebra. Now, by Proposition \ref{prop.baseHilbert
quasi}-(ii) we have that $A$ is isomorphic to a  graded iterated
Ore extension $R[x_1;\theta_1]\cdots [x_{n};\theta_{n}]$. It is
known that if $A$ is a Calabi-Yau algebra of dimension $d$, then the
global dimension of $A$ is $d$ (see for example \cite{Berger3},
Remark 2.8). Then, using Proposition \ref{prop. ASKoszulCY}-(i) and
applying induction on $n$ we obtain that $A$ is a Calabi-Yau algebra
of  dimension $d+n$.
\end{proof}

\begin{corollary}\label{cor.AS impl skewCal}
Let $R$ be an Artin-Schelter regular algebra of global dimension
$d$. Then every graded quasi-commutative skew PBW extension
$A=\sigma(R)\langle x_1,\dots, x_n\rangle$ is skew Calabi-Yau of
global dimension $d+n$.
\end{corollary}

There are two notions of Nakayama automorphisms: one for skew Calabi-
Yau algebras and one for Frobenius algebras. In this
paper, we focus ourselves on skew Calabi-Yau algebras, or equivalently, AS-
regular algebras in the connected graded case (Proposition \ref{prop.lema1.2}). In this case, the two notions
of Nakayama automorphisms will coincide in the sense of the Koszul duality (see \cite{Zhu}, Proposition 1.4). Let $B = T(V )/\langle R\rangle$ be a Koszul algebra.  For a graded automorphism $\sigma$ of $B$, we define a map $\sigma^*: V^*\to V^*$ by $\sigma^*(f)(x) = f(\sigma(x))$, for
each $f\in V^*$ and $x\in V$. Note that $\sigma^*$ induces a graded automorphism of
$B^{!}$ because $\sigma$ is assumed to preserve the relation space $R$. We still use the notation
$\sigma^*$ for this algebra automorphism (see \cite{Zhu}). Suppose that $\{x_1, x_2,\dots, x_n\}$ is a $\mathbb{K}$-linear basis
of $V$ and $\{x_1^*, x_2^*,\dots, x_n^*\}$ is the corresponding dual basis of $V^*$. If $\sigma(x_i)=\sum_jc_{ij}x_j$, for $c_{ij}\in \mathbb{K}$, $1\leq i,j\leq n$, then we have $\sigma^*(x_i^*)= \sum_jc_{ji}x_j^*$. Moreover, for each $i, j = 1, 2,\dots, n$, we have  $\sigma^*(x_i^*)(x_j) = x_i^*(\sigma(x_j))$. Let $B$ be a Koszul AS-regular algebra of dimension $d$. Then, the Nakayama automorphism $\nu$ of $B$ is equal to $\epsilon^{d+1}\mu^*$, where $\mu$ is the Nakayama automorphism of the Frobenius algebra $B^!$ and $\epsilon$ is the automorphism of $B$ defined by $a\mapsto (-1)^{deg \ a}$, for each homogeneous element $a\in B$ (\cite{Zhu}, Proposition 1.4). Let $B$ be a Koszul AS-regular algebra of global dimension $d$, with Nakayama automorphism $\nu$. Suppose that $\sigma$ is a graded automorphism of $B$ and $\sigma^{*}$ is its corresponding dual graded automorphism of the dual algebra $B^{!}$. The \emph{homological determinant}, denoted \rm{hdet}, is a homomorphism from the graded automorphism group $GrAut(B)$ of $B$ to the multiplicative group $\mathbb{K}\setminus\{0\}$ such that $\sigma^{*}(u) = ({\rm{hdet}\sigma})u$, for any $u\in Ext_B^d(\mathbb{K}, \mathbb{K})$ (\cite{Wu}, Proposition 1.11).

\begin{proposition}[\cite{Zhu}, Proposition 3.15]\label{prop.3.15} Suppose that $R$ is a Koszul AS-regular algebra with Nakayama
automorphism $\nu$, $\sigma$ is a graded algebra automorphism of $R$ and $A = R[x; \sigma]$. The Nakayama automorphism $\mu$ of $A$ is given by:
\[
\mu(a)= \left\{
  \begin{array}{ll}
    \sigma^{-1}\nu(a), & a\in R; \\
    {\rm{hdet}}(\sigma)a, & a=x.
  \end{array}
\right.
\]
\end{proposition}

\begin{theorem}[\cite{Liu1}, Theorem 3.3]\label{Theorem 3.3} Let $K$ be a unital commutative ring and let $R$ be a projective $K$-algebra and $A = R[x; \sigma, \delta]$ be an Ore extension. Suppose that $R$ is an skew Calabi-Yau algebra of dimension $d$ with Nakayama automorphism $\nu$. Then $A$ is skew Calabi-Yau of dimension $d + 1$ and the Nakayama automorphism $\nu'$ of $A$ satisfies that $\nu'|_R = \sigma^{-1}\nu$ and $\nu'(x) = ux + b$ with $u$, $b\in A$ and $u$ invertible.
\end{theorem}

\begin{remark}[\cite{Liu1}, Remark 3.4]\label{rem3.4}
 $\nu'(x) = x + b$ if $\sigma = Id$, and $\nu'(x) = ux$ if $\delta = 0$.
\end{remark}

Let $A = \mathbb{K}\langle x, y\rangle/(yx - xy - x^2)$ be the Jordan plane, which is an AS-regular
algebra of dimension 2. Note that $A = \mathbb{K}[x][y; \delta_1]$ with $\delta_1(x) = x^2$. It follows that
$A$ is skew Calabi-Yau, with the Nakayama automorphism
given by $\nu(x) = x$ and $\nu(y) = 2x + y$. Thus, $A$ is not Calabi-Yau. On one hand, $B = A[z; \nu]$ is an Ore extension of Jordan plane. Then $A$ is
skew Calabi-Yau with the Nakayama automorphism $\nu'$ such that $\nu'(x) = x$ and
$\nu'(y) = y$. On the other hand, $A = \mathbb{K}[x, z][y; \delta]$ where $\delta$ is given by $\delta(x) = x^2$ and $\delta(z) =-2xz$. So, $\nu'(z) = z$.
It follows that $A$ is Calabi-Yau (see \cite{Liu1}).


\begin{theorem}\label{prop.Nakayama}
Let $R$ be an  Artin-Schelter regular algebra with Nakayama
automorphism $\nu$. Then the Nakayama automorphism $\mu$ of a graded
quasi-commutative skew PBW extension $A=\sigma(R)\langle x_1, \dots,
x_n\rangle$  is  given by
\begin{align*}
\mu(r) & = (\sigma_1\cdots\sigma_n)^{-1}\nu(r), \text{ for } r\in R, \text{ and}\\
\mu(x_i) & =u_i\prod_{j=i}^nc_{i,j}^{-1}x_i, \text{ for each } 1\leq
i\leq n,
\end{align*}
where $\sigma_i$ is as in Proposition \ref{sigmadefinition}, $u_i,
c_{i,j}\in \mathbb{K}\setminus \{0\}$, and the elements $c_{i,j}$
are as in Definition \ref{def.skewpbwextensions}.
\end{theorem}

\begin{proof}
Note that $A$ is skew Calabi-Yau (see Corollary \ref{cor.AS impl
skewCal}) and therefore the Nakayama automorphism of $A$  exists. By
Proposition \ref{prop.baseHilbert quasi}-(ii) and its proof
 we have that  $A$ is iso\-mor\-phic to a graded iterated Ore extension $R[x_1; \theta_1]\cdots[x_n; \theta_n]$, where $\theta_i$ is bijective;
 $\theta_1=\sigma_1$; $$\theta_j: R[x_1;\theta_1]\cdots[x_{j-1}; \theta_{j-1}]\to R[x_1;\theta_1]\cdots[x_{j-1}; \theta_{j-1}]$$ is
 such that $\theta_j(x_i)=c_{i,j}x_i$ ($c_{i,j}\in \mathbb{K}$ as in Definition \ref{def.skewpbwextensions}), $1 \leq i < j \leq n$
 and $\theta_i(r)=\sigma_i(r)$, for $r\in R$. Note that $\theta_j^{-1}(x_i)=c_{i,j}^{-1}x_i$. 
Now, since $R$ is connected then by Remark \ref{rem.prop of graded
skew}, $A$ is connected. So,  the multiplicative group of $R$ and
also the multiplicative group of $A$ is $\mathbb{K}\setminus \{0\}$,
therefore  the identity
map is the only inner automorphism of $A$. Let $\mu_i$ the Nakayama automorphism of $R[x_1;\theta_1]\cdots[x_{i};\theta_{i}]$.\\
 
By Theorem \ref{Theorem 3.3} and Remark  \ref{rem3.4} we have that the Nakayama automorphism $\mu_1$ of $R[x_1; \theta_1]$ is given by  $\mu_1(r)=\sigma_1^{-1}\nu(r)$ for $r\in R$,  and $\mu_1(x_1) = u_1x_1$ with $u_1\in \mathbb{K}\setminus \{0\}$; the Nakayama automorphism $\mu_2$ of $R[x_1; \theta_1][x_2; \theta_2]$ is given by  $\mu_2(r)=\sigma_2^{-1}\mu_1(r)=\sigma_2^{-1}\sigma_1^{-1}\nu(r)$, for $r\in R$;  $\mu_2(x_1) = \theta_2^{-1}\mu_1(x_1)= \theta_2^{-1}(u_1x_1)=u_1\theta_2^{-1}(x_1)=u_1c_{1,2}^{-1}x_1$ and  $\mu_2(x_2)= u_2x_2$, for $u_2\in \mathbb{K}\setminus \{0\}$; the Nakayama automorphism $\mu_3$ of $R[x_1; \theta_1][x_2; \theta_2][x_3; \theta_3]$ is given by  $\mu_3(r)=\sigma_3^{-1}\mu_2(r)=\sigma_3^{-1}\sigma_2^{-1}\sigma_1^{-1}\nu(r)$, for $r\in R$; $\mu_3(x_1)= \theta_3^{-1}\mu_2(x_1)=  \theta_3^{-1}(u_1c_{1,2}^{-1}x_1)=u_1c_{1,2}^{-1}\theta_3^{-1}(x_1)= u_1c_{1,2}^{-1}c_{1,3}^{-1}x_1$; $\mu_3(x_2)= \theta_3^{-1}\mu_2(x_2)=\theta_3^{-1}(u_2x_2)=u_2c_{2,3}^{-1}x_2$ and $\mu_3(x_3)= u_3x_3$, for $u_3\in \mathbb{K}\setminus \{0\}$.\\
Continuing with the procedure we have that the Nakayama automorphism
of $A$ is given by
$\mu(r)=\mu_n(r)=\sigma_n^{-1}\cdots\sigma_2^{-1}\sigma_1^{-1}\nu(r)$, for $r\in R$; $\mu(x_1)= u_1c_{1,2}^{-1}c_{1,3}^{-1}\cdots
c_{1,n}^{-1} x_1$; $\mu(x_2)=u_2c_{2,3}^{-1}\cdots c_{2,n}^{-1}x_2$.
In general, for $1\leq i\leq n$, we have that
$\mu(x_i)=u_ic_{i,i+1}^{-1}c_{i,i+2}^{-1}\cdots
c_{i,n}^{-1}x_i=u_i(c_{i,n}\cdots c_{i,i+2}c_{i,i+1})^{-1}x_i$, for
$u_i\in \mathbb{K}\setminus \{0\}$. Note that $c_{i,i}=1$ (Remark
\ref{rem.grad f y rep}-(i)).
\end{proof}

\begin{example}\label{ex.quantum affine}  Let   $A=\mathcal{O}_{\textbf{q}}(\mathbb{K}^n)$ be the quantum affine $n$-space.
$A=\mathcal{O}_{\textbf{q}}(\mathbb{K}^n)$ is a graded
quasi-commutative skew PBW extension of $\mathbb{K}[x_1]$ (see
Exam\-ple \ref{ex.quasi}-\ref{mult param quant aff}), with
$\sigma_j(k)=k$ for $k\in\mathbb{K}$ and $\sigma_j(x_1)= q_{1j}x_1$,
$j\geq 2$. Therefore, according to Proposition \ref{prop.baseHilbert
quasi} and its proof, $A$ is isomorphic to a graded iterated Ore
extension $\mathbb{K}[x_1][x_2;\theta_2]\cdots [x_n;\theta_n]$,
where $\theta_j(k)=k$ for $k\in \mathbb{K}$ and
$\theta_j(x_i)=q_{ij}x_i$, for $1\leq i<j\leq n$. Note that the
Nakayama automorphism $\nu$ of $\mathbb{K}[x_1]$ is the identity
map.  Applying Theorem \ref{prop.Nakayama} we have that the Nakayama
automorphism $\mu$ of $A$ is given by $\mu(k)=k$ for $k\in
\mathbb{K}$, $\mu(x_1)= (\sigma_1\cdots\sigma_n)^{-1}\nu(x_1)=
(q_{1n}^{-1}\cdots q_{12}^{-1})x_1= (q_{n1}\cdots q_{21})x_1$, and
$\mu(x_i) =u_iq_{i(i+1)}^{-1}q_{i(i+2)}^{-1}\cdots q_{in}^{-1}
x_i=u_iq_{(i+1)i}q_{(i+2)i}\cdots q_{ni} x_i=
u_iq_{(i+1)i}q_{(i+2)i}\cdots q_{ni} x_i$, for each $2\leq i\leq n$.
Since  $\mu$ is unique up to an inner automorphism (see Proposition
\ref{Nak.uniq}) and the invertible elements in
$\mathcal{O}_{\textbf{q}}(\mathbb{K}^n)$ are those nonzero scalars
in $\mathbb{K}$, the identity map is the only inner automorphism of
$\mathcal{O}_{\textbf{q}}(\mathbb{K}^n)$. Therefore, using the same
reasoning of \cite{Liu1} in the proof of Proposition 4.1, we have
that $u_i=q_{1i}q_{2i}\cdots q_{(i-1)i}$. Then, $\mu(x_i) =
(\prod_{j=1}^{n}q_{ji})x_i$, for $2 \leq i\leq n$.
\end{example}

\begin{example}\label{ex.iterated skew polynomial} Let $R$ be an Artin-Schelter regular algebra of global dimension $d$,
with Nakayama automorphism $\nu$. Let $A= R[x_1,\dots, x_n;
\sigma_1,\dots, \sigma_n]$ be an iterated skew polynomial ring (see
\cite{Goodearl}, page 23-24), with $\sigma_i$ graded. $A$ is a skew
PBW extension of $R$ with relations $x_ir=\sigma_i(r)x_i$ and
$x_jx_i=x_ix_j$, for $r\in R$ and $1\leq i,j\leq n$.  As $R$ is
graded and $c_{i,j}=1\in R_0$, then by Proposition \ref{prop.quasi
is grad} we have that $A$ is a graded quasi-commutative skew PBW
extension of $R$. Therefore, $A$ is a skew Calabi-Yau algebra
(Corollary \ref{cor.AS impl skewCal}).  By Proposition
\ref{prop.baseHilbert quasi}, $R[x_1,\dots, x_n; \sigma_1,\dots,
\sigma_n]\cong R[x_1; \theta_1]\cdots [x_n;\theta_n]$, where
$\theta_j(r)=\sigma_j(r)$ and $\theta_j(x_i)= x_i$ for $i< j$.
Applying Theorem \ref{prop.Nakayama} we have that the Nakayama
automorphism $\mu$ of $A$ is given by $\mu(r)
=(\sigma_1\cdots\sigma_n)^{-1}\nu(r)$, if $r\in R$ and
$\mu(x_i)=u_i\prod_{j=i}^nc_{i,j}^{-1}x_i= u_ix_i$, $u_i\in
\mathbb{K}\setminus \{0\}$, $1\leq i\leq n$.
\end{example}

\begin{example}\label{ex.nak q-dilat oper} The algebra of linear partial q-dilation operators $H$ is a graded
quasi-commu\-tative skew PBW extension of $\mathbb{K}[t_1, \dots,
t_n]$ (see Exam\-ple \ref{ex.quasi}-\ref{ex.q-dilat oper}).
According to Proposition \ref{sigmadefinition}, we have that $\sigma_j(t_i) = t_i\text{ for } i\neq j, \quad 1\leq i \leq n,
\quad 1\leq j \leq m$;  $\sigma_j(t_i) = qt_i \text{ for } i=j,\quad 1 \leq i,j \leq m$; $\delta_j = 0, \text{ for } 1\leq j\leq m$.  By Proposition \ref{prop.baseHilbert quasi}, $H$  is isomorphic to a
graded iterated Ore extension of endomorphism type 
$\mathbb{K}[t_1,\dots, t_n][H_1^{(q)}; \theta_1] \cdots
[H_{m-1}^{(q)}; \theta_{m-1}]H_m^{(q)}; \theta_m]$, 
with $\theta_j(t_i) = t_i\text{ for } i\neq j, \quad 1\leq i \leq n, \quad 1\leq j \leq m$, $\theta_j(t_i) = qt_i \text{ for }
i=j,\quad 1 \leq i, j \leq m$; $\theta_j(H^{(q)}_i) = H^{(q)}_i
\text{ for } 1\leq i, j\leq m$. Since $\mathbb{K}[t_1,\dots, t_n]$
is a Calabi-Yau algebra, then its Nakayama automorphism $\nu$ is the
identity map. Applying Theorem \ref{prop.Nakayama}, we have that the
Nakayama automorphism $\mu$ of $H$ is given by
\begin{align*}
&\mu(t_i) =\sigma_m^{-1}\cdots\sigma_1^{-1}(t_i)=\left\{
                                         \begin{array}{ll}
                                           \sigma_i^{-1}(t_i)=q^{-1}t_i, & 1\leq i\leq m; \\
                                           t_i, & m<i\leq n,
                                         \end{array}
                                       \right.\\
&\mu(H^{(q)}_j)=H^{(q)}_j, \text{ for } 1\leq j\leq m.
\end{align*}
\end{example}

\begin{remark} If in Example \ref{ex.iterated
skew polynomial} $R$ is also a Koszul algebra, then due to the
Theorem \ref{teo.Kosz impl Kosz}-(ii), we have that $A$ is a Koszul
algebra.  The Nakayama automorphism $\mu$ of Example
\ref{ex.iterated skew polynomial} is then $\mu(r)
=(\sigma_1\cdots\sigma_n)^{-1}\nu(r)$, if $r\in R$ and $\mu(x_i)=
({\rm{hdet}}\sigma_i) x_i$, $1\leq i\leq n$ (see \cite{Zhu}, Theorem
4.6). Since  $\mu$ is unique up to an inner automorphism (see
Proposition \ref{Nak.uniq}) and the invertible elements in
$R[x_1,\dots, x_n; \sigma_1,\dots, \sigma_n]$ are those nonzero
scalars in $\mathbb{K}$, the identity map is the only inner
automorphism of $R[x_1,\dots, x_n; \sigma_1,\dots, \sigma_n]$.
Therefore, $({\rm{hdet}}\sigma_i)=u_i$.
\end{remark}


\begin{thebibliography}{30}

\bibitem{Artamonov1} V. A. Artamonov, ``Derivations of skew PBW extensions", {\em Commun. Math. Stat.}, vol. 3, no. 4,  449-457, 2015.

\bibitem{Berger3} R. Berger and R. Taillefer,  ``Poincare-Birkhoff-Witt
deformations of Calabi-Yau algebras",  {\em J. Noncommut. Geom.}, vol. 1, pp.  241-270, 2007.

\bibitem{Brown} K. A. Brown  and J. J. Zhang,  ``Dualising complexes and twisted Hochschild (co)homology for Noetherian Hopf algebras", {\em J. Algebra}, vol. 320, pp. 1814-1850, 2008.

\bibitem{Yesica} N. R. Gonz\'alez and Y. P. Su\'arez, ``Ideales en el anillo de polinomios torcidos $R[x;\sigma,\delta]$", {\em Revista Ciencia en Desarrollo}, vol. 5, no. 1, pp. 31-37, 2014.



\bibitem{Goodearl} K. R. Goodearl and  R. B. Warfield, {\em An introduction to noncommutative Noetherian rings},
London Mathematical Society Student Texts, 61. Cambridge University Press, Cambridge, 2004.


\bibitem{Goodman} J. Goodman and U. Kr\"ahmer, ``Untwisting a twisted Calabi-Yau
Algebra", {\em J. Algebra}, vol. 406, pp.  271-289, 2014.

\bibitem{LezamaGallego} C. Gallego and O. Lezama, ``Gr\"obner bases for ideals of $\sigma$-PBW extensions", {\em Comm.  Algebra}, vol. 39, no. 1, pp. 50-75, 2011.

\bibitem{He5} J. W. He, F. Van Oystaeyen and  Y. Zhang,   ``Skew polynomial algebras
with coefficients in Koszul Artin-Schelter regular algebras", {\em
J. Algebra}, vol. 390,  pp. 231-249, 2013.

\bibitem{LezamaAcostaReyes2015} O. Lezama, J. P. Acosta and A. Reyes, ``Prime ideals of skew PBW extensions", {\em Rev. Un. Mat. Argentina}, vol. 56, no. 2, pp. 39-55, 2015.

\bibitem{LezamaReyes} O. Lezama and A. Reyes, ``Some Homological Properties of Skew PBW Extensions", {\em  Comm. Algebra}, vol. 42, pp.  1200-1230, 2014.

\bibitem{Liu1} L.-Y. Liu,  S. Wang  and   Q.-S. Wu, ``Twisted Calabi-Yau property of Ore extensions", {\em J. Noncommut. Geom.}, vol. 8, no. 2, pp.   587-609, 2014.

\bibitem{ReyesSuarez2106-2} A. Reyes and H.  Su\'arez, ``Armendariz property for skew PBW extensions and their classical ring of quotients", {\em Rev. Integr. Temas Mat.}, vol. 34, no. 2, pp.  147-168, 2016.

\bibitem{ReyesSuarez2106-3} A. Reyes and H. Su\'arez, ``A note on zip and reversible skew PBW extensions", {\em Bol. Mat.}, vol. 23, no. 1, pp.  71-79, 2016.

\bibitem{ReyesSuarez2106-4} A. Reyes and H. Su\'arez, ``Some remarks about the cyclic homology of skew PBW extensions", {\em Ciencia en Desarrollo}, vol. 7, no. 2, pp.  99-107, 2016.

\bibitem{ReyesSuarez2106-5} A. Reyes and H. Su\'arez, ``Bases for quantum algebras and skew Poincar\'e-Birkhoff-Witt extensions", {\em Momento, Rev. Fis.},  vol. 54, no. 2, pp. 54-75, 2017.

\bibitem{ReyesSuarez2107-4} A. Reyes and H. Su\'arez, ``Enveloping algebra and
Calabi-Yau algebras over skew Poincar\'e-Birkhoff-Witt extensions",
{\em Far East J. Math. Sci. (FJMS)}, vol. 102, no. 2, pp. 373-397,
2017.
\bibitem{ReyesSuarez2017-3} A. Reyes and H. Su\'arez, ``A notion of compatibility for Armendariz and Baer properties
over skew PBW extensions", {\em Rev. Un. Mat. Argentina}, vol. 59,
no. 1, pp. 157-178, 2018.

\bibitem{ReyesSuarez2017-2} A. Reyes and H. Su\'arez, ``PBW bases for some
3-dimensional skew polynomial algebras", {\em Far East J. Math. Sci.
(FJMS)}, vol. 101, no. 6, pp. 1207-1228, 2017.

\bibitem{ReyesSuarez2017-1} A. Reyes and H. Su\'arez, ``$\sigma$-PBW extensions of skew
Armendariz rings",  {\em Adv. Appl. Clifford Algebr.}, vol. 27, pp.
3197-3224, 2017.

\bibitem{Reyes} M. Reyes, D. Rogalski and J. J. Zhang, ``Skew Calabi-Yau algebras and homological identities", {\em Adv. in Math.}, vol.
 264, pp.  308-354, 2014.

\bibitem{Suarez} H. Su\'arez, ``Koszulity for graded skew PBW extensions", {\em Comm.  Algebra}, vol. 45, no. 10, pp.
4569-4580, 2017.

\bibitem{SuarezLezamaReyes2015} H. Su\'arez, O.  Lezama and A. Reyes, ``Some relations between $N$-Koszul, Artin-Schelter regular and Calabi-Yau algebras with skew PBW extensions", {\em Ciencia en Desarrollo}, vol. 6, no. 2, pp.  205-213, 2015.

\bibitem{SuarezLezamaReyes2107-1} H. Su\'arez, O. Lezama, and A. Reyes.
``Calabi-Yau property for graded skew PBW extensions", {\em To
appear in Rev. Colombiana Mat.}, 2017.

\bibitem{SuarezReyes2016} H. Su\'arez and A. Reyes, ``A generalized Koszul property for skew PBW estensions", {\em Far East J. Math. Sci.}, vol. 101,  no. 2, pp.  301-320, 2017.

\bibitem{SuarezReyes7} H. Su\'arez and A. Reyes, ``Koszulity for skew PBW
extensions over fields", {\em JP J. Algebra Number Theory Appl.},
vol. 39, no. 2, pp. 181-203, 2017.

\bibitem{Venegas2015} C. Venegas, ``Automorphisms for skew PBW extensions and skew quantum polynomial rings", {\em Comm.  Algebra}, vol. 43, no. 5, pp.  1877-1897, 2015.


\bibitem{Wu} Q.-S. Wu, C. Zhu, ``Skew group algebras of Calabi-Yau algebras", {\em J. Algebra}, vol. 340, pp.  53-76, 2011.

\bibitem{Zhu} C. Zhu, F. Van Oystaeyen and Y. Zhang, ``Nakayama automorphism of double Ore extensions of Koszul regular algebras",
{\em Manuscripta math.}, vol. 152, no. 3-4, pp. 555-584, 2017.


\end{thebibliography}
\end{document}